\documentclass[11pt]{amsart}
\usepackage{tikz-cd}
\usepackage{amscd}
\usepackage{amsmath, amssymb, comment}
\usepackage{amsfonts}
\usepackage{enumerate}
\usepackage{color}
\usepackage{url}
\newcommand{\de}{\partial}

\newcommand{\dbar}{\overline{\partial}}

\newcommand{\Ric}{\mathrm{Ric}}

\newcommand{\vol}{\mathrm{Vol}}

\newcommand{\ve}{\varepsilon}

\renewcommand{\leq}{\leqslant}
\renewcommand{\geq}{\geqslant}

\newcommand{\be}{\begin{equation}}
\newcommand{\ee}{\end{equation}}

\begin{document}
\newcounter{remark}
\newcounter{theor}
\setcounter{theor}{1}
\newtheorem{claim}{Claim}
\newtheorem{theorem}{Theorem}[section]
\newtheorem{lemma}[theorem]{Lemma}
\newtheorem{corollary}[theorem]{Corollary}
\newtheorem{conjecture}[theorem]{Conjecture}
\newtheorem{proposition}[theorem]{Proposition}
\newtheorem{question}{Question}[section]
\newtheorem{definition}[theorem]{Definition}
\newtheorem{remark}[theorem]{Remark}

\numberwithin{equation}{section}

\title[Rigidity of eigenvalues]{The rigidity of eigenvalues on K\"ahler manifolds with positive Ricci lower bound}

\author[J. Chu]{Jianchun Chu}
\address[Jianchun Chu]{School of Mathematical Sciences, Peking University, Yiheyuan Road 5, Beijing 100871, People's Republic of China}
\email{jianchunchu@math.pku.edu.cn}

\author[F. Wang]{Feng Wang}
\address[Feng Wang]{School of Mathematical Sciences, Zhejiang University, Hangzhou 310013, People's Republic of China}
\email{wfmath@zju.edu.cn}

\author[K. Zhang]{Kewei Zhang}
\address[Kewei Zhang]{School of Mathematical Sciences, Beijing Normal University, Beijing 100875, People's Republic of China}
\email{kwzhang@bnu.edu.cn}

\begin{abstract}
In this work, optimal rigidity results for eigenvalues on K\"ahler manifolds with positive Ricci lower bound are established. More precisely, for those K\"ahler manifolds whose first eigenvalue agrees with the Ricci lower bound, we show that the complex projective space is the only one with the largest multiplicity of the first eigenvalue. Moreover, there is a specific gap between the largest and the second largest multiplicity. In the K\"ahler--Einstein case, almost rigidity results for eigenvalues are also obtained.
\end{abstract}

\maketitle

\section{Introduction}

\subsection{Background}

Comparison theorems in Riemannian geometry are indispensable and play significant roles in the solution of various problems in geometry analysis; see the textbook \cite{PeterLiBook12} for an exposition of this classical topic.

There has been great interest in extending Riemannian comparison theorems to the K\"ahler setting (see e.g. \cite{LW05,CN05,L11,TY12,L14,W15,BB17,LY18,NZ18,F18,L21,Z22,DSS21,DS23} and the references therein).
Due to the special feature of K\"ahler geometry, comparison theorems in this setting could behave differently from the Riemannian case.

For instance, Liu \cite{L14} proved that for a K\"ahler manifold $(M,g,J)$ of real dimension $m>2$ with $\Ric(g)\geq (m-1)g$, one has $\vol(M,g)\leq\big(1-\varepsilon(m)\big)\vol(\mathbb S^m,g_{\mathbb S^m})$, where $\varepsilon(m)>0$ is a dimensional constant.
Namely, the volume upper bound given by the Bishop--Gromov volume comparison is not sharp in the K\"ahler setting. Recently, the third-named author derived the optimal volume upper bound for K\"ahler manifolds with positive Ricci curvature, which generalizes \cite{BB17,F18}.
\begin{theorem}$(\text{\rm{Zhang} \cite{Z22}})\textbf{.}$
    Let $(M,\omega)$ be a compact K\"ahler manifold of complex dimension $n$ with $\Ric(\omega) \geq \omega$. Then one has $\vol(M,\omega) \leq \vol(\mathbb{CP}^n,\omega_{\mathbb{CP}^n})$, and the equality holds if and only if $(M,\omega)$ is biholomorphically isometric to $(\mathbb{CP}^n,\omega_{\mathbb{CP}^n})$. Here $\omega_{\mathbb{CP}^n}$ denotes the K\"ahler--Einstein metric on $\mathbb{CP}^n$ so that $\Ric(\omega_{\mathbb{CP}^n})=\omega_{\mathbb{CP}^n}$.
\end{theorem}

See \cite{W15,Z22,DSS21} for further developments regarding the almost rigidity of volume. To prove these results one crucially uses techniques arising from the recent developments in the K-stability theory, which are completely different from those used in the Riemannian setting.

Similarly, a K\"ahlerian diameter comparison theorem is proved in \cite{LW05,DS23}.

\begin{theorem}$(\text{\rm{Li--Wang} \cite{LW05}}, \text{\rm{Datar--Seshadri} \cite{DS23}})\textbf{.}$
    Let $(M,\omega)$ be a compact K\"ahler manifold of complex dimension $n$ whose bisectional curvature satisfies $BK\geq\frac{1}{n+1}$. Then $\mathrm{diam}(M,\omega)\leq\mathrm{diam}(\mathbb{CP}^n,\omega_{\mathbb{CP}^n})$. The equality holds if and only if $(M,\omega)$ is biholomorphically isometric to $(\mathbb{CP}^n,\omega_{\mathbb{CP}^n})$.
\end{theorem}

Li--Wang \cite{LW05} established the diameter inequality. The equality case was proved by Datar--Seshadri \cite{DS23} whose proof used pluripotential theory, which is again different from the classical Riemannian setting; compare also \cite{TY12,LY18}, whose proof uses the geometry of submanifolds.

These new comparison theorems reveal special features of K\"ahler geometry, making the research in the direction even more appealing.
In this paper we are interested in the eigenvalue comparison theorems on K\"ahler manifolds with positive Ricci lower bound.

To begin with, we first recall the Riemannian setting.
Let $(M,g)$ be a compact Riemannian manifold of real dimension $m$. In this paper, we assume that all the manifolds considered are connected. One of the most important differential operators on $(M,g)$ is the Laplace-Beltrami operator, which is defined by
\[
\Delta_{\mathbb{R}}u = \frac{1}{\sqrt{\det g}}\frac{\de}{\de x^{i}}\left(\sqrt{\det g} \, g^{ij}\frac{\de u}{\de x^{j}}\right), \ \ \ \text{for $u\in C^{\infty}(M,\mathbb{R})$}.
\]
The first eigenvalue $\lambda_1$ of $\Delta_{\mathbb{R}}$ can be characterized as
\[
\lambda_1 = \inf\left\{\frac{\int_{M}|\nabla u|^{2}dg}{\int_{M}u^{2} dg}~\Big|~u\in W^{1,2}(M,g), \ \int_{M}u dg= 0 \right\},
\]
where $W^{1,2}(M,g)$ denotes the Sobolev space of $(M,g)$.

It is well-known that the $m$-dimensional unit sphere $(\mathbb{S}^{m},g_{\mathbb{S}^{m}})$ has constant sectional curvature one (which implies $\Ric(g_{\mathbb{S}^{m}})=(m-1)g_{\mathbb{S}^{m}}$) and $\lambda_1=m$. Then $(\mathbb{S}^{m},g_{\mathbb{S}^{m}})$ becomes a ``comparison model" in Riemannian geometry. Recall the following classical estimate for the first eigenvalue.

\begin{theorem}$(\text{\rm{Lichnerowicz} \cite{Li58}})\textbf{.}$
Let $(M,g)$ be a compact Riemannian manifold of real dimension $m$ with $\Ric(g)\geq (m-1)g$. Then $\lambda_1\geq m$.
\end{theorem}

One can further characterize the equality case, obtaining the rigidity for the first eigenvalue on manifolds with positive Ricci lower bound.

\begin{theorem}\label{Riemannian rigidity}$(\text{\rm{Obata} \cite{Ob62}})\textbf{.}$\label{rigidity Riemannian case}
Let $(M,g)$ be a compact Riemannian manifold of real dimension $m$ with $\Ric(g)\geq (m-1)g$. If $\lambda_1=m$, then $(M,g)$ is isometric to $(\mathbb{S}^{m},g_{\mathbb{S}^{m}})$.
\end{theorem}

The goal of this paper is to establish analogous results for K\"ahler manifolds. As we shall see, new phenomenon occurs in the K\"ahler setting, requiring a different way of approaching the problem.

\subsection{Eigenvalue comparison and rigidity}
Let $(M,\omega)$ be a compact K\"ahler manifold of complex dimension $n$ and write $\omega=\sqrt{-1}g_{i\bar{j}}dz^{i}\wedge d\bar{z}^{j}$. As usual, we consider the complex Laplace operator defined by
\[
\Delta u = \frac{n\sqrt{-1}\de\dbar u\wedge\omega^{n-1}}{\omega^{n}} = g^{i\bar{j}}\de_{i}\de_{\bar{j}} u, \ \ \ \text{for $u\in C^{\infty}(M,\mathbb{R})$}.
\]
The relationship between two Laplace operators is given by $\Delta_{\mathbb{R}}=2\Delta$. Denote the eigenvalues of $\Delta$ by
\[
0 < \lambda_{1} \leq \lambda_{2} \leq \ldots \leq \lambda_{i} \leq \ldots
\]

A reasonable counterpart of $\mathbb{S}^{m}$ in K\"ahler geometry is $\mathbb{CP}^{n}$, which can be equipped with a standard K\"ahler--Einstein metric $\omega_{\mathbb{CP}^{n}}$ that is proportional to Fubini--Study metric so that $\Ric(\omega_{\mathbb{CP}^{n}})=\omega_{\mathbb{CP}^{n}}$. For $(\mathbb{CP}^{n},\omega_{\mathbb{CP}^{n}})$, one has $\lambda_1=...=\lambda_{n^2+2n}=1$ and $\lambda_{n^2+2n+1}>1$ (see Lemma \ref{lem:Cpn-multiplicity}). The next theorem is well-known, which can be viewed as an eigenvalue comparison theorem in K\"ahler geometry.

\begin{theorem}$(\emph{see e.g. } \text{\rm{Futaki} \cite[Theorem 2.4.5]{Fu88}})\textbf{.}$\label{lambda 1 estimate}
Let $(M,\omega)$ be a compact K\"ahler manifold of complex dimension $n$ with $\Ric(\omega)\geq \omega$. Then $\lambda_{1}\geq 1$.
\end{theorem}

Motivated by Theorem \ref{Riemannian rigidity}, in the equality case $\lambda_{1}=1$, it is very natural to expect that $(M,\omega)$ is biholomorphically isometric to $(\mathbb{CP}^{n},\omega_{\mathbb{CP}^{n}})$. However, it is well-known that this fails in the K\"ahler setting. In fact, the first eigenvalue of any K\"ahler--Einstein Fano manifold admitting a nontrivial holomorphic vector field is equal to one \cite{M57}. An explicit counterexample is as follows (see Section \ref{examples} for details): for positive integers $k$, $l$ such that $k+l=n$ and constant $a\in[0,1)$, consider
\begin{equation}\label{counterexample}
(M,\omega) =
(\mathbb{CP}^{k},\omega_{\mathbb{CP}^{k}})\times\big(\mathbb{CP}^{l},(1-a)\omega_{\mathbb{CP}^{l}}\big).
\end{equation}
It can be shown that $\Ric(\omega)\geq \omega$ and $\lambda_{1}=...=\lambda_{k^2+2k}=1$, but $M\ncong\mathbb{CP}^n$. In the special case where $a=0$ the metric $\omega$ is even K\"ahler--Einstein with $\Ric(\omega)=\omega$ and $\lambda_1=...=\lambda_{k^{2}+2k+l^{2}+2l}=1$, but still $M\ncong\mathbb{CP}^n$.

These examples suggest that, to establish rigidity results for eigenvalues in the K\"ahler setting, we need to take multiplicities of eigenvalues into account as well. In what follows we always assume $n\geq2$, since a one-dimensional K\"ahler manifold with positive Ricci curvature must be $\mathbb{CP}^1$ and hence the rigidity result follows from Theorem \ref{rigidity Riemannian case}. Here is our first main result.

\begin{theorem}\label{CPn rigidity}
Let $(M,\omega)$ be a compact K\"ahler manifold of complex dimension $n$ with $\Ric(\omega)\geq \omega$. If $\lambda_{n^2+3}=1$, then
$(M,\omega)$ is biholomorphically isometric to $(\mathbb{CP}^{n},\omega_{\mathbb{CP}^{n}})$.
\end{theorem}

Theorem \ref{CPn rigidity} substantially improves \cite[Theorem 1.9]{L18}, where the author proved the result under the stronger assumptions that $\Ric(\omega)=\omega$ and $\lambda_{n^2+2n}=1$. We remove his K\"ahler--Einstien condition and reduce the number of required eigenvalues. Moreover,
our result is optimal, since the conclusion fails if we only assume $\lambda_{n^2+2}=1$. Indeed, if we choose $k=n-1$, $l=1$ and $a=0$ in \eqref{counterexample}, i.e.,
\[
(M,\omega) =
(\mathbb{CP}^{n-1},\omega_{\mathbb{CP}^{n-1}})\times(\mathbb{CP}^{1},\omega_{\mathbb{CP}^{1}}),
\]
then
$$
\Ric(\omega)=\omega\text{ and }\lambda_{1}=\ldots=\lambda_{n^{2}+2}=1,
$$
while $M\ncong\mathbb{CP}^n$. Therefore, Theorem \ref{CPn rigidity} gives a sharp characterization of $\mathbb{CP}^n$ in terms of the Ricci curvature and eigenvalues.

Theorem \ref{CPn rigidity} also shows that $(\mathbb{CP}^{n},\omega_{\mathbb{CP}^{n}})$ has the largest multiplicity $n^2~+~2n$ of the first eigenvalue among all K\"ahler manifolds with $\Ric(\omega)\geq\omega$ and $\lambda_1=1$, and there is nothing else with multiplicity in the range $[n^2+3,n^2+2n)$.
In other words, there is a specific gap between the largest and second largest multiplicity. We refer the reader to Section \ref{sec:2&3-large-mult} for further classification on the K\"ahler manifolds attaining the second and third largest multiplicity.

To prove Theorem \ref{CPn rigidity}, we use the viewpoint of twisted K\"ahler--Einstein metrics and carry out a careful infinitesimal analysis of the isometric actions on the manifold. During the course of the proof, we establish a new sharp criterion for the constancy of the holomorphic sectional curvature for general almost Hermitian manifolds (see Proposition \ref{dimension prop 2}), which plays a key role in our work and could be of independent interest.

\subsection{Almost rigidity}

We are also interested in the almost rigidity of eigenvalues. In the Riemannian case, Petersen \cite{P99} and Aubry \cite{A05} showed the following.

\begin{theorem}$(\text{\rm{Petersen} \cite{P99}}, \text{\rm{Aubry} \cite{A05}})\textbf{.}$
    For all $\varepsilon>0$, there exists $\delta=\delta(m,\varepsilon)>0$ such that the following holds. Let $(M,g)$ be a Riemannian manifold of dimension $m$ with $\Ric(g)\geq (m-1)g$ and $\lambda_{m}(M,g)< m+\delta$, then $M$ is diffeomorphic to $\mathbb S^m$ and
    $$
    d_{\mathrm{GH}}\big((M,g),(\mathbb{S}^m,g_{\mathbb{S}^m})\big)<\varepsilon.
    $$
\end{theorem}
Petersen \cite{P99} first proved this result by assuming $\lambda_{m+1}<m+\delta$, and then Aubry \cite{A05} improved it.
But the result fails if one only assumes $\lambda_{m-1}(M,g)< m+\delta$ (see \cite{A05}).
We expect that the K\"ahlerian analogue of the above result holds as well.

\begin{conjecture}
For all $\varepsilon>0$, there exists $\delta=\delta(n,\varepsilon)>0$ such that the following holds. Let $(M,\omega)$ be a compact K\"ahler manifold of complex dimension $n$ with $\Ric(\omega)\geq \omega$ and $\lambda_{n^2+3}< 1+\delta$, then $M$ is biholomorphic to $\mathbb{CP}^n$ and
$$
d_{\mathrm{GH}}\big((M,\omega),(\mathbb{CP}^n,\omega_{\mathbb{CP}^n})\big)<\varepsilon.
$$
\end{conjecture}

To attack this conjecture, it seems that new ideas are required. It is even unclear to us if the volume of K\"ahler manifolds satisfying $\Ric(\omega)\geq \omega$ and $\lambda_{n^2+3}\leq 1+\delta$ could collapse or not. However, we can prove this conjecture under the K\"ahler--Einstein assumption.

\begin{theorem}\label{thm1}
There exists a dimensional constant $\varepsilon=\varepsilon(n)>0$ such that the following holds. Let $(M,\omega)$ be a compact K\"ahler manifold of complex dimension $n$ with $\Ric(\omega)=\omega$. If $\lambda_{n^2+3}< 1+\varepsilon$, then $(M,\omega)$ is biholomorphically isometric to $(\mathbb{CP}^n,\omega_{\mathbb{CP}^n})$.
\end{theorem}

In other words, if a K\"ahler--Einstein manifold $(M,\omega)$ is not biholomorphic to $\mathbb{CP}^n$, then $\lambda_{n^2+3}\geq 1+\varepsilon(n)$ for some dimensional constant $\varepsilon(n)>0$.
More generally, if we drop the K\"ahler--Einstein condition but keep assuming  $\omega\in 2\pi c_1(M)$, we are able to show the following almost rigidity  result.

\begin{theorem}\label{thm2}
For any $\varepsilon>0$ there exists $\delta=\delta(\varepsilon,n)>0$
such that the following holds. Let $(M,\omega)$ be a compact K\"ahler manifold of complex dimension $n$ with $\omega\in2\pi c_1(M)$. Assume that $$\Ric(\omega)\geq(1-\delta)\omega \text{ and }\lambda_{n^2+3}\leq 1+\delta,$$
then $M$ is biholomorphic to $\mathbb{CP}^n$ and
$$
d_{\mathrm{GH}}\big((M,\omega),(\mathbb{CP}^n,\omega_{\mathbb{CP}^n})\big)<\varepsilon.
$$
\end{theorem}

The above two results are proved using  deep analysis of the Gromov--Hausdorff limit spaces of (almost) K\"ahler--Einstein manifolds that arises in the solution of the Yau--Tian--Donaldson conjecture; see Section \ref{sec:conv} for details.

\medskip

The paper is organized as follows: In Section 2, we collect some standard results about eigenfunctions. In Section 3, we investigate the isometry group action on almost Hermitian manifolds and then prove Theorem \ref{CPn rigidity}. In Section 4, we establish the convergence of eigenfunctions in the Gromov--Hausdorff topology. In Section 5, we give the proof of Theorem \ref{thm1} and \ref{thm2}. In Section 6, some explicit examples are computed, which indicate that our main results are optimal. In Section \ref{sec:2&3-large-mult}, we use the classification results in \cite{I07,IK09} to further sharpen our main results.

\medskip

{\bf Acknowledgments.} The authors thank Lifan Guan, Wenshuai Jiang and Jun Yu for helpful discussions. J.C. was partially supported by National Key R\&D Program of China 2023YFA1009900, NSFC grant 12271008 and the Fundamental Research Funds for the Central Universities, Peking University. F.W. was partially supported by NSFC grant 12031017 and NSF of Zhejiang Province for Distinguished Young Scholars grant LR23A010001. K.Z. was partially supported by NSFC grants 12101052, 12271040, and 12271038.

\section{Eigenfunctions and automorphisms}

In this part we collect some standard results regarding eigenfunctions on K\"ahler manifolds (cf. \cite{M57}). For the reader's convenience we will reproduce the proofs.
\begin{lemma}\label{eigenfunction lemma}
Let $(M,\omega)$ be a compact K\"ahler manifold of complex dimension $n$ with $\Ric(\omega)\geq\omega$ and $J$ be the complex structure. Set $\theta=\Ric(\omega)-\omega$. If $u$ is a eigenfunction of $1$, then
\begin{enumerate}\setlength{\itemsep}{1mm}
\item[(a)] $\nabla^{1,0}u=g^{i\bar{j}}u_{\bar{j}}\de_{i}$ is a non-zero holomorphic vector field, or equivalently $L_{J\nabla u}J=0$;
\item[(b)] $\iota_{\nabla^{1,0}u}\theta=0$, or equivalently $\iota_{J\nabla u}\theta=0$;\vspace{1mm}
\item[(c)] $\int_{M}\theta^{n}=0$;\vspace{1.5mm}
\item[(d)] $J\nabla u$ is a Killing field.
\end{enumerate}
\end{lemma}

\begin{proof}
We compute
\[
\begin{split}
\int_{M}u_{i}u^{i}\omega^{n}
= {} & -\int_{M}(\Delta u)_{i}u^{i}\omega^{n} \\
= {} & -\int_{M}(g^{k\bar{l}}u_{k\bar{l}})_{i}u^{i}\omega^{n} \\
= {} & -\int_{M}g^{k\bar{l}}(u_{ki\bar{l}}-R_{k\bar{j}i\bar{l}}u^{\bar{j}})u^{i}\omega^{n} \\
= {} & -\int_{M}g^{k\bar{l}}u_{ki\bar{l}}u^{i}\omega^{n}+\int_{M}R_{i\bar{j}}u^{i}u^{\bar{j}}\omega^{n} \\
= {} & \int_{M}u_{ki}u^{ki}\omega^{n}+\int_{M}u_{i}u^{i}\omega^{n}+\int_{M}\theta_{i\bar{j}}u^{i}u^{\bar{j}}\omega^{n}.
\end{split}
\]
Combining this with $\theta\geq0$, we have
\begin{equation*}
\int_{M}u_{ki}u^{ki}\omega^{n} = 0, \ \ \
\int_{M}\theta_{i\bar{j}}u^{i}u^{\bar{j}}\omega^{n}=0
\end{equation*}
and so
\begin{equation*}
u_{ij}=0, \ \ \ u_{\bar i \bar j}=0, \ \ \ \theta_{i\bar{j}}u^{i}u^{\bar{j}} = 0.
\end{equation*}
Then we obtain (a). Using $\theta\geq0$ again, $\theta_{i\bar{j}}u^{i}u^{\bar{j}}=0$ implies (b).

For (c), If $\nabla^{1,0}u(p)\neq0$ at $p\in M$, then (b) shows $\theta^{n}(p)=0$. By (a),
we know the set $\{\nabla^{1,0}u=0\}$ is a subvariety and hence of measure zero. It follows that
\[
\int_{M}\theta^{n} = \int_{\{\nabla^{1,0}u\neq0\}}\theta^{n} = 0,
\]
which is (c).

Using (a), we have $L_{J\nabla u}J=0$. To show (d), it then suffices to show $L_{J\nabla u}\omega=0$.
By $\omega(\cdot,\cdot)=g(J\cdot,\cdot)$ we have that
\[
\iota_{J\nabla u}\omega=-g(\nabla u,\cdot)=-du.
\]
Then by Cartan's formula,
\[
L_{J\nabla u}\omega = d\,\iota_{J\nabla u}\omega+\iota_{J\nabla u}d\omega = d\,\iota_{J\nabla u}\omega = -d^{2}u = 0,
\]
which completes the proof.
\end{proof}

As an immediate consequence of Lemma \ref{eigenfunction lemma}, we obtain the following rigidity result of K\"ahler--Einstien metrics, generalizing slightly \cite[Theorem 3.1]{TY12}.

\begin{proposition}\label{KE rigidity big}
Let $(M,\omega)$ be a compact K\"ahler manifold of complex dimension $n$ with $\Ric(\omega)\geq \omega$.
If $\lambda_{1}=1$ and $[\Ric(\omega)-\omega]=c[\alpha]$ for some constant $c$ and big class $[\alpha]\in H^{1,1}(X,\mathbb{R})$, then $\Ric(\omega)= \omega$.
\end{proposition}

\begin{proof}
Using Lemma \ref{eigenfunction lemma} (c), we obtain
\[
c^{n}\int_{M}\alpha^{n} = \int_{M}\theta^{n} = 0,
\]
which implies $c=0$. Combining $\theta\geq0$ and $[\theta]=0$, we obtain $\theta=0$ by the strong maximum principle, and hence $\Ric(\omega)=\omega$.
\end{proof}

\begin{lemma}\label{isomorphism}
Let $(M,\omega)$ be a compact K\"ahler manifold of complex dimension $n$ with $\Ric(\omega)\geq\omega$, $\mathrm{Iso}(M,\omega)$ be the group of biholomorphic isometries of $(M,\omega)$ and $\mathfrak{iso}(M,\omega)$ be the Lie algebra of $\mathrm{Iso}(M,\omega)$. Set
\[
\Lambda_{1}(M,\omega) = \{u\in C^{\infty}(M,\mathbb{R})~|~\text{$u$ is an eigenfunction of $1$ on $M$} \},
\]
\[
\mathfrak{h}(M,\omega) = \{W\in\mathfrak{iso}(M,\omega)~|~\iota_{W}\theta=0 \}.
\]
Then $\Lambda_{1}(M,\omega)$ is isomorphic to $\mathfrak{iso}(M,\omega)$ as linear spaces through the correspondence given by $F(u)=J\nabla u$. Moreover, there is a subgroup of $\mathrm{Iso}(M,\omega)$, denoted by $H(M,\omega)$, such that $\mathfrak{h}(M,\omega)$ is the Lie algebra of $H(M,\omega)$.
\end{lemma}

\begin{proof}
By Lemma \ref{eigenfunction lemma} (a), (b) and (d), $F$ maps $\Lambda_{1}(M,\omega)$ into $\mathfrak{h}(M,\omega)$ injectively. We show this map is also onto. Given any $W\in\mathfrak{h}(M,\omega)$, we know that $V=JW+\sqrt{-1}W$ is a holomorphic vector field on $M$. We claim that there exist $u_{1},u_{2}\in\Lambda_{1}(M,\omega)$ such that
\begin{equation}\label{isomorphism claim}
V = \nabla^{1,0}(u_{1}+\sqrt{-1}u_{2}).
\end{equation}
Given this claim, we obtain
\[
JW+\sqrt{-1}W = (\nabla u_{1}+J\nabla u_{2})+\sqrt{-1}(-J\nabla u_{1}+\nabla u_{2}).
\]
Considering the imaginary part of both sides and using Lemma \ref{eigenfunction lemma} (d), we obtain $\nabla u_{2}=W+J\nabla u_{1}$ is a Killing field and so $\nabla^{2}u_{2}=L_{\nabla u_{2}}g=0$. Combining this with $u_{2}\in\Lambda_{1}(M,\omega)$, we obtain $u_{2}=-\Delta u_{2}=0$. It then follows that $W=J\nabla(-u_{1})$, which shows that the map $F$ is onto.

Now we prove the claim \eqref{isomorphism claim}. Using Hodge theory and the fact that $H^{1}(X,\mathbb{C})=0$ (as $M$ is Fano), one can find $u\in C^\infty(M,\mathbb{C})$ such that $\nabla^{1,0}u=V$. Since $V$ is holomorphic, then $u_{\bar j \bar k}=0$. Next we argue that $\Delta u+u$ is actually a holomorphic function on $M$, hence $\Delta u+u=c$ for some constant $c\in\mathbb{C}$, thus both $\mathrm{Re}(u-c)$ and $\mathrm{Im}(u-c)$ belong to $\Lambda_{1}(M,\omega)$, which will finish the proof of the claim \eqref{isomorphism claim}. So we compute
    \begin{equation*}
        \begin{aligned}
             (\Delta u+u)_{\bar k}&=g^{i\bar j}u_{i\bar j \bar k}+u_{\bar k}\\
             &=g^{i\bar j}u_{\bar j i \bar k}+u_{\bar k}\\
             &=g^{i\bar j}u_{\bar j \bar k i}-g^{i\bar j}R_{i\bar k}u_{\bar j}+u_{\bar k}\\
             &=0-g^{i\bar j}g_{i\bar k}u_{\bar j}-g^{i\bar j}\theta_{i\bar k}u_{\bar j}+u_{\bar k}\\
             &=-\theta_{i\bar k}V^i=0.
        \end{aligned}
    \end{equation*}
Therefore, $\Delta u+u$ is holomorphic, as claimed.

Finally we show the existence of $H(M,\omega)$. It suffices to show that $\mathfrak{h}(M,\omega)$ is closed under Lie bracket. For any $W_{1},W_{2}\in\mathfrak{h}(M,\omega)$, we know that $V_{i}=JW_{i}+\sqrt{-1}W_{i}$ $(i=1,2)$ are holomorphic vector fields on $M$. Therefore, $[V_{1},\bar{V}_{2}]=[\bar{V}_{1},V_{2}]=0$ and so
\[
\begin{split}
[W_{1},W_{2}] = {} & \left[\frac{V_{1}-\bar{V}_{1}}{2\sqrt{-1}},\frac{V_{2}-\bar{V}_{2}}{2\sqrt{-1}}\right] \\[1mm]
= {} & -\frac{1}{4}\left([V_{1},V_{2}]+\overline{[V_{1},V_{2}]}-[V_{1},\bar{V}_{2}]-[\bar{V}_{1},V_{2}]\right) \\[1mm]
= {} & -\frac{1}{4}\left([V_{1},V_{2}]+\overline{[V_{1},V_{2}]}\right).
\end{split}
\]
To show $[W_{1},W_{2}]\in\mathfrak{h}(M,\omega)$, it then suffices to show $\iota_{[V_{1},V_{2}]}\theta=0$. Using $\iota_{W_{i}}\theta=0$ and that $\theta$ is compatible with $J$, we see that $\iota_{V_{i}}\theta=0$. Then Cartan's formula shows
\[
L_{V_{i}}\theta= d\,\iota_{V_{i}}\theta+\iota_{V_{i}}d\theta = 0,
\]
which implies
\[
L_{[V_{1},V_{2}]}\theta = L_{V_{1}}\circ L_{V_{2}}\theta-L_{V_{2}}\circ L_{V_{1}}\theta = 0.
\]
On the other hand, $[V_{1},V_{2}]$ is also a holomorphic vector fields on $M$. Using Hodge theory, one has $\iota_{[V_{1},V_{2}]}\theta=\sqrt{-1}\bar\partial f$ for some $f\in C^\infty(X,\mathbb C)$. Using Cartan's formula again,
\[
L_{[V_{1},V_{2}]}\theta= d\,\iota_{[V_{1},V_{2}]}\theta+\iota_{[V_{1},V_{2}]}d\theta = \sqrt{-1}\partial\bar{\partial}f.
\]
Then we obtain $\sqrt{-1}\partial\bar{\partial}f=0$ and so $f$ is constant. This implies $\iota_{[V_{1},V_{2}]}\theta=0$, as desired.
\end{proof}

\section{Isometry group and Rigidity}
\subsection{Almost Hermitian manifolds with large isometry group} To prove Theorem \ref{CPn rigidity}, we shall use the general philosophy that the geometry of a manifold becomes rather rigid if the dimension of the isometry group is large enough; see also \cite{Y53,K67,T69,I07,IK09} for related discussions. Although we are mainly concerned with K\"ahler manifolds, in this part we will prove some general results for almost Hermitian manifolds, which could be of independent interest.

Let $(M,\omega)$ be an almost Hermitian manifold, namely, it is a Riemannian manifold $(M,g)$ equipped with an almost complex structure $J$ that is compatible with $g$, and $\omega(\cdot,\cdot):=g(J\cdot,\cdot)$ is the corresponding K\"ahler form. $(M,\omega)$ is a K\"ahler manifold precisely when $J$ is integrable and $d\omega=0$. We say that an almost Hermitian manifold $(M,\omega)$ has constant holomorphic sectional curvature if there is a constant $c$ such that for any $p\in M$ and any unit vector $X\in T_p^{1,0}M$ one has
$$
R(X,\bar X,X,\bar X)=c,
$$
where $R$ denotes the Riemannian curvature tensor associated with $g$.

In what follows we always assume that $M$ is connected and we denote by $\mathrm{Iso}(M,\omega)$ the group of differomorphisms of $M$ that preserve both $g$ and $J$. Note that $\mathrm{Iso}(M,\omega)$ is a Lie group, and will be called the biholomorphic isometry group of the almost Hermitian manifold.

We first recall a rigidity result in \cite{T69}. We reproduce its proof, as it is instructive for our discussions below.

\begin{proposition}\label{dimension prop 1}
Let $(M,\omega)$ be an almost Hermitian manifold of real dimension $2n$.
Then one has
\[
\dim_{\mathbb{R}}\mathrm{Iso}(M,\omega) \leq n^{2}+2n.
\]
If $\dim_{\mathbb{R}}\mathrm{Iso}(M,\omega)=n^{2}+2n$, then $(M,\omega)$ has constant holomorphic sectional curvature.
\end{proposition}

\begin{proof}
Consider the unitary frame bundle $FM$ of $(M,\omega)$ which consists of all $(p,\{e_{i}\}_{i=1}^{n})$, where $p\in M$ and $\{e_{i}\}_{i=1}^{n}$ is a unitary frame of $T_{p}^{1,0}M$. In other words, $FM$ is a principle $U(n)$-bundle over $M$. Since the real dimension of unitary group $U(n)$ is $n^{2}$, then $FM$ is a manifold of real dimension $n^{2}+2n$.

Consider the group action $\mathrm{Iso}(M,\omega) \times FM \to FM$ by
\[
\phi \cdot (p,\{e_{i}\}_{i=1}^{n}) = \big(\phi(p),\{d\phi(e_{i})\}_{i=1}^{n}\big).
\]
By considering minimal geodesics emanating from a point we see that the above action is free. Then each orbit is an injectively immersed submanifold of $FM$ that is diffeomorphic to $\mathrm{Iso}(M,\omega)$. Therefore
\[
\dim_{\mathbb{R}}\mathrm{Iso}(M,\omega) \leq \dim_{\mathbb{R}}FM = n^{2}+2n.
\]
When $\dim_{\mathbb{R}}\mathrm{Iso}(M,\omega)=n^{2}+2n$, using the connectedness of $FM$, each orbit is equal to $FM$. This shows that the above action is transitive. Then $(M,\omega)$ has constant holomorphic sectional curvature.
\end{proof}

The next proposition is key to us, which improves the previous result and gives an optimal characterization of the constancy of holomorphic sectional curvature in terms of the dimension of the biholomorphic isometry group.

\begin{proposition}\label{dimension prop 2}
Let $(M,\omega)$ be an almost Hermitian manifold of real dimension $2n$, $n\geq2$, and $G\subseteq \mathrm{Iso}(M,\omega)$ be a Lie subgroup. If $\dim_{\mathbb{R}}G\geq n^{2}+3$, then $(M,\omega)$ has constant holomorphic sectional curvature.
\end{proposition}

This result fails if we only assume $\dim_{\mathbb{R}}G\geq n^{2}+2$, since one has
\[
\dim_{\mathbb{R}}\mathrm{Iso}\big((\mathbb{CP}^{n-1},\omega_{\mathbb{CP}^{n-1}})\times(\mathbb{CP}^1,\omega_{\mathbb{CP}^1})\big)= n^{2}+2,
\]
but $\mathbb{CP}^{n-1}\times\mathbb{CP}^1$ cannot have constant holomorphic sectional curvature.

\begin{proof}
By Lemma \ref{transitive lemma} below, we know that the Lie group action $G\times M\to M$ is transitive. It then suffices to show that the holomorphic sectional curvature of all complex lines in $T_{p}^{1,0}M$ is constant at any given point $p\in M$.

Denote the orbit and stabilizer of the action by $O$ and $S$:
\[
O = \{\phi(p)~|~\phi\in G\}, \ \ \
S = \{\phi\in G~|~\phi(p) = p\}.
\]
Note that $S$ is a closed Lie subgroup of $G$ and
\begin{equation}\label{dimension prop 2 eqn 1}
\dim_{\mathbb{R}}S = \dim_{\mathbb{R}}G-\dim_{\mathbb{R}}O \geq n^{2}+3-2n.
\end{equation}
For each $\phi\in S$, the tangent map $d\phi:T_{p}^{1,0}M\to T_{p}^{1,0}M$ at $p$ can be regarded as an isometry of $(\mathbb{C}^{n},\omega_{\mathrm{Euc}})$. We claim that the action of $S$ on $T^{1,0}_{p}M$ is faithful. Indeed, if $d\phi=\mathrm{id}$ for some isometry $\phi\in S$, then $\phi=\mathrm{id}$ by the uniqueness of geodesics.
Therefore we can identify $S$ with
\[
T := \{d\phi:T_{p}^{1,0}M\to T_{p}^{1,0}M~|~\phi\in S\},
\]
which is a Lie subgroup of $U(n)$. Then the Lie groups
\[
H := T/\{\xi I_{n}\in T~|~\xi\in\mathbb{C},\ |\xi|=1\}
\]
and
\[
PU(n)=U(n)/\{\xi I_{n}~|~\xi\in\mathbb{C},\ |\xi|=1\} \cong \mathrm{Iso}\big(\mathbb{P}(T^{1,0}_{p}M),\omega_{\mathbb{P}(T^{1,0}_{p}M)}\big)
\]
fit into the following commutative diagram:
\[ \begin{tikzcd}
T \arrow[hookrightarrow]{r}{} \arrow[twoheadrightarrow]{d}{} & U(n) \arrow[twoheadrightarrow]{d}{} \\
H \arrow[hookrightarrow]{r}{}  & PU(n)
\end{tikzcd}
\]
Since each element in $\mathbb{P}(T^{1,0}_{p}M)$ corresponds to a complex line in $T_{p}^{1,0}M$. To prove the holomorphic sectional curvature of all complex lines in $T_{p}^{1,0}M$ is constant, we show that the induced action of $H$ on $\mathbb{P}(T^{1,0}_{p}M)$ is transitive. Using \eqref{dimension prop 2 eqn 1}, we obtain
\[
\dim_{\mathbb{R}}H \geq \dim_{\mathbb{R}}T-1 = \dim_{\mathbb{R}}S-1 \geq n^{2}+3-2n-1 = (n-1)^{2}+1.
\]
Then the transitivity follows from Lemma \ref{transitive lemma}.
\end{proof}

\begin{lemma}\label{transitive lemma}
Let $(M,\omega)$ be an almost Hermitian manifold of real dimension $2n$ and $G\subset\mathrm{Iso}(M,\omega)$ be a Lie subgroup satisfying $\dim_{\mathbb{R}}G\geq n^{2}+1$, then the group action $G\times M\to M$ is transitive.
\end{lemma}

\begin{proof}
Fix a point $p\in M$, and consider the orbit and stabilizer:
\[
O = \{\phi(p)~|~\phi\in G\}, \ \ \
S = \{\phi\in G~|~\phi(p) = p\}.
\]
We argue by contradiction. Assume that the action is not transitive. Then
\[
k:=\dim_{\mathbb{R}}O\leq 2n-1.
\]
We still identify $S$ with
\[
T = \{d\phi:T_{p}M\to T_{p}M~|~\phi\in S\}.
\]
Then
\[
\dim_{\mathbb{R}}G = \dim_{\mathbb{R}}O+\dim_{\mathbb{R}}S = k+\dim_{\mathbb{R}}T.
\]

Define two spaces $V_{1}$ and $V_{2}$ by
\[
V_{1} = T_{p}O, \ \ V_{2} = (T_{p}O)^{\perp}.
\]
Since $d\phi$ is an isometry, it preserves $V_{1}$ and $V_{2}$. Let $W_{1}$ and $W_{2}$ be the orthogonal complement of $V_{1}\cap JV_{1}$ in $V_{1}$ and $V_{2}\cap JV_{2}$ in $V_{2}$ respectively, where $J$ denotes the almost complex structure. Then
\[
T_{p}M = V_{1}\oplus V_{2} = \big((V_{1}\cap JV_{1})\oplus W_{1}\big)\oplus \big((V_{2}\cap JV_{2})\oplus W_{2}\big).
\]
Observe that $JW_{1}\subset W_{1}\oplus W_{2}$ and $W_{1}\cap JW_{1}=\{0\}$. So the orthogonal projection map $JW_1\to W_2$ is injective, and hence
\[
\dim_{\mathbb{R}}W_{1}=\dim_{\mathbb R}JW_1\leq\dim_{\mathbb{R}}W_{2}.
\]
Interchanging $W_{1}$ and $W_{2}$, we obtain that
$\dim_{\mathbb{R}}W_{1}=\dim_{\mathbb{R}}W_{2}$, and thus the composed map $W_1\to JW_1\to W_2$, which we denote by $\mu$, yields an isomorphism
$$W_1\overset{\mu}{\cong}W_2.$$

We claim that $d\phi$ preserves $V_{1}\cap JV_{1}$, $V_{2}\cap JV_{2}$, $W_{1}$ and $W_{2}$.
Indeed, for any $v\in V_{1}\cap JV_{1}$, we have $v\in V_{1}$ and $Jv\in V_{1}$, which implies
\[
d\phi(v)\in V_{1}, \ \
J(d\phi(v)) = d\phi(Jv)\in V_{1}
\]
and so
\[
d\phi(v)\in V_{1}\cap JV_{1}.
\]
This shows that $d\phi$ preserves $V_{1}\cap JV_{1}$. Since $d\phi$ acts isometrically on $V_1$, $d\phi$ also preserves $W_{1}$. That $d\phi$ preserves $V_{2}\cap JV_{2}$ and $W_{2}$ follows in the same way. From here we further deduce that $d\phi$ is $\mu$-equivariant, namely, one has
$
d\phi(\mu(v))=\mu(d\phi(v))
$
for any $v\in W_1$.
So the action of $d\phi$ on $W_2$ is completely determined by that on $W_1$.

Since $V_{1}\cap JV_{1}$ is $J$-invariant, its real dimension must be even. Set $\dim_{\mathbb{R}}(V_{1}\cap JV_{1})=2r$ and then
\[
\dim_{\mathbb{R}}W_{1} = \dim_{\mathbb{R}}W_{2} = k-2r, \ \ \dim_{\mathbb{R}}(V_{2}\cap JV_{2}) = 2(n-k+r).
\]
It is clear that
\[
\max(0,k-n) \leq r \leq \frac{k}{2}.
\]
We will derive a contradiction by showing that
$$\dim_{\mathbb{R}}G\leq n^{2}.$$
The key point is that each $d\phi\in T$ is completely determined by its action on $V_{1}\cap JV_{1}$, $V_{2}\cap JV_{2}$ and $W_{1}$.
The argument will be split into two cases.

\bigskip
\noindent
{\bf Case 1.} $k\leq2n-2$.
\bigskip

In this case, we have
\[
\begin{split}
\dim_{\mathbb{R}}T \leq {} & \dim_{\mathbb{R}}U(r)+\dim_{\mathbb{R}}U(n-k+r)+\dim_{\mathbb{R}}O(k-2r) \\
= {} & r^{2}+(n-k+r)^{2}+\frac{1}{2}(k-2r)(k-2r-1).
\end{split}
\]
Regard the above as a quadratic polynomial of $r$. Direct calculation shows the maximum is achieved at $r=\frac{k}{2}$, then
\[
\dim_{\mathbb{R}}G = k+\dim_{\mathbb{R}}T \leq k+\left(\frac{k}{2}\right)^{2}+\left(n-\frac{k}{2}\right)^{2} \leq n^{2}.
\]

\bigskip
\noindent
{\bf Case 2.} $k=2n-1$.
\bigskip

In this case, we have $r=n-1$. Then
\[
\dim_{\mathbb{R}}T \leq \dim_{\mathbb{R}}U(n-1)+\dim_{\mathbb{R}}O(1)= (n-1)^{2}
\]
and so
\[
\dim_{\mathbb{R}}G \leq 2n-1+(n-1)^{2} = n^{2}.
\]

\end{proof}

\subsection{Proof of rigidity}
Now we are in the position to prove Theorem \ref{CPn rigidity}.

\begin{proof}[Proof of Theorem \ref{CPn rigidity}]
Using Proposition \ref{dimension prop 2}, $(M,\omega)$ has constant holomorphic sectional curvature. Since the underlying manifold $M$ is simply connected (as $M$ is Fano), then Theorem \ref{CPn rigidity} follows from \cite[Theorem 7.9]{KN69}.

\end{proof}

For later use, we prove the following result.

\begin{proposition}\label{KE rigidity}
Let $(M,\omega)$ be a compact K\"ahler manifold of complex dimension $n$ with $\Ric(\omega)\geq\omega$, $\lambda_{n^2+1}=1$, then $\Ric(\omega)=\omega$.
\end{proposition}

\begin{proof}
Let $H:=H(M,\omega)$ be the Lie subgroup of $\mathrm{Iso}(M,\omega)$ whose Lie algebra is $\mathfrak{h}(M,\omega)$ in Lemma \ref{isomorphism}. Then we obtain
\[
\dim_{\mathbb{R}}H = \dim_{\mathbb{R}}\Lambda_{1}(M,\omega) \geq n^{2}+1.
\]
Let $FM$ be the unitary frame bundle of $(M,\omega)$ which consists of all $(p,\{e_{i}\}_{i=1}^{n})$, where $p\in M$ and $\{e_{i}\}_{i=1}^{n}$ is a unitary frame of $T_{p}^{1,0}M$.
Consider the group action $H\times FM \to FM$. Fix $(p,\{e_{i}\}_{i=1}^{n})$, define
\[
\Phi: H \to FM, \ \
\Phi(\phi) = \big(\phi(p),\{d\phi(e_{i})\}_{i=1}^{n}\big).
\]
The image of $\Phi$ is actually the orbit $O_{(p,\{e_i\}^n_{i=1})}$ under the action of $H$ on $FM$. It is clear that $O_{(p,\{e_i\}^n_{i=1})}$ is a submanifold of $FM$ and $\pi(O_{(p,\{e_i\}^n_{i=1})})=O_p$, where $\pi: FM\rightarrow M$ is the bundle projection and $O_p$ is the orbit of $p$ under the action of $H$ on $M$, i.e.
\[
O_{p} = \{\phi(p)~|~\phi\in H\}.
\]
The tangent map $d\Phi:\mathfrak{h}(M,\omega)\rightarrow T_{(p,\{e_i\}^n_{i=1})}FM$ is injective and $\mathrm{Im}(d\Phi)=T_{(p,\{e_i\})}O_{(p,\{e_i\}^n_{i=1})}$. Since $\pi$ is a submersion, $d\pi\circ\Phi:\mathfrak{h}(M,\omega)\rightarrow T_p O_p$ is onto. By Lemma \ref{transitive lemma}, we know that $O_p=M$. It follows that $d\pi\circ \Phi: \mathfrak{h}(M,\omega)\rightarrow T_p M$ is onto. Then we obtain
\[
\mathrm{Span}\{J\nabla u~|~\text{$u$ is an eigenfunction of eigenvalue 1}\}=T_p M.
\]
Combining this with Lemma \ref{eigenfunction lemma} (b), we obtain $\theta=0$ at $p$. Since $p$ is arbitrary, we know that $\theta=0$ and so $\Ric(\omega)=\omega$.
\end{proof}

\section{convergence of eigenfunctions}
\label{sec:conv}
In this section, we will prove some convergence results of eigenfunctions in the Gromov--Hausdorff topology. We will consider two cases: the sequence of K\"ahler--Einstein metrics and the sequence of almost K\"ahler--Einstein metrics. The general results are proved in \cite{CC00b}.

At first, let $(M_i,\omega_i)$ be a sequence of K\"ahler manifolds with $\Ric(\omega_i)=\omega_i$. Note that the sequence $(M_i,\omega_i)$ is volume non-collapsing, since
\[
\int_{M_i}\omega^n_i=(2\pi)^nc_1(M_i)^n\geq (2\pi)^n.
\]
Here we used the algebraic fact that $c_1(M_i)^n$ is a positive integer. By Cheeger--Colding theory \cite{CC97}, up to a subsequence, we get
$$
(M_i,\omega_i)\xrightarrow{\mathrm{GH}}(Z,d),
$$
where $(Z,d)$ is a compact metric space. It admits a singular-regular decomposition
$$
Z=\mathcal{R}\cup\mathcal{S},
$$
where $\mathcal{R}$ is an open convex subset in $Z$ by \cite{CC97,CN12}.  By \cite{T13,DS14}, $Z$ is homeomorphic to an $n$-dimensional $\mathbb Q$-Fano variety equipped with a (singular) K\"ahler--Einstein metric $\omega_\infty$, whose restriction on $\mathcal{R}$ is a smooth K\"ahler--Einstein metric and the metric completion of $(\mathcal{R},\omega_{\infty}|_{\mathcal{R}})$ coincides with $(Z,d)$. Furthermore, the singular locus $\mathcal{S}$ agrees with algebro-geometric singular locus of $Z$ when viewed as a variety, which has complex codimension at least two.

If $u_i$ is a sequence of functions on $(M_i,\omega_i)$ satisfying $\Delta_i u_i=-\lambda_i u_i$ and $\int_{M_i} u_i^2\omega_i^n=1$, where $\lambda_i\rightarrow 1$, then $u_i$ are uniformly Lipschitz by Cheng-Yau's gradient estimate \cite{CY75}. So after possibly passing to a subsequence, $u_i$ converges to a Lipschitz function $u$ on $Z$. Since $\omega_i\rightarrow \omega_\infty$ smoothly on $\mathcal R$, by the standard elliptic regularity theory, $u$ is smooth and satisfies
\begin{equation}\label{limit function on R}
\Delta_{\omega_\infty} u=-u \text { on } \mathcal{R}.
\end{equation}

\begin{lemma}\label{Killing}
$J \nabla u$ is a Killing vector field on $\mathcal R$ and can be extended to be the imaginary part of a holomorphic vector field on $Z$.
\end{lemma}

\begin{proof}
As in the proof of \cite[Lemma 3.4]{TW20}, we can prove the result using cut-off functions. For any $\ve>0$, denoting
$$T_\ve(\mathcal{S})=\{x\in Z~|~\operatorname{dist}(x, \mathcal{S}) \leq \ve\},$$
there is a cut-off function $\gamma_\ve \in C_0^{\infty}(Z \backslash \mathcal{S})$ such that
$$
\int_{\mathcal{R}}\left|\nabla \gamma_\ve\right|^2 \leq \ve, \ \
\gamma_\ve \equiv 0 \text { in } T_\ve(\mathcal{S}).
$$
Using \eqref{limit function on R}, we can compute as in Lemma \ref{eigenfunction lemma} to get
\[
\begin{split}
\int_{\mathcal R}\gamma_\ve^2 u_{i}u^{i}\omega^{n}
= {} & -\int_{\mathcal R}\gamma_\ve^2g^{k\bar{l}}u_{ki\bar{l}}u^{i}\omega^{n}+ \int_{\mathcal R}\gamma_\ve^2 R_{i\bar{j}}u^{i}u^{\bar{j}}\omega^{n} \\
= {} & \int_{\mathcal R}\gamma_\ve^2u_{ki}u^{ki}\omega^{n}+2 \int_{\mathcal R}\gamma_\ve g^{k\bar{l}}u_{ki}u^{i}{(\gamma_\ve)}_{\bar l}\omega^{n}+\int_{\mathcal R}\gamma_\ve^2 u_{i}u^{i}\omega^{n}.
\end{split}
\]
Recall $u$ is a Lipschitz function on $Z$. Then $|\nabla u|\leq C$ on $\mathcal{R}$ for some constant $C$. Combining this with the Cauchy-Schwarz inequality,
\[
2\int_{\mathcal R}\gamma_\ve g^{k\bar{l}}u_{ki}u^{i}{(\gamma_\ve)}_{\bar l}\omega^{n}
\geq -\delta\int_{\mathcal R}\gamma^2_\ve u_{ki}u^{ki}\omega^{n}-\frac{C}{\delta} \int_{\mathcal{R}}\left|\nabla \gamma_\ve\right|^2 \omega^{n}.
\]
It follows that
\[
\int_{\mathcal R}\gamma_\ve^2 u_{i}u^{i}\omega^{n}
\geq (1-\delta)\int_{\mathcal R}\gamma_\ve^2u_{ki}u^{ki}\omega^{n}
-\frac{C}{\delta} \int_{\mathcal{R}}\left|\nabla \gamma_\ve\right|^2 \omega^{n}
+\int_{\mathcal R}\gamma_\ve^2 u_{i}u^{i}\omega^{n}.
\]
Letting $\ve \rightarrow 0$ and then $\delta\rightarrow 0$, we get
\[
\int_{\mathcal R} u_{i}u^{i}\omega^{n} \geq \int_{\mathcal R}u_{ki}u^{ki}\omega^{n}+ \int_{\mathcal R}u_{i}u^{i}\omega^{n}.
\]
Hence $\int_{\mathcal R}u_{ki}u^{ki}\omega^{n}=0$ on $\mathcal R$. Then the same argument of Lemma \ref{eigenfunction lemma} shows that $J \nabla u$ is a Killing vector field on $\mathcal R$. By the proof of \cite[Lemma 6.9]{T15}, $J \nabla u$ can be extended to be the imaginary part of a holomorphic vector field on $Z$.
\end{proof}
The above lemma also holds when $(M_i,\omega_i)$ is a sequence of almost K\"ahler--Eintein manifolds. We recall the definition for the reader's convenience.

\begin{definition}$(\text{\rm{Tian--Wang} \cite{TW15}})\textbf{.}$
    A sequence of pointed compact K\"ahler manifolds $(M_i,\omega_i,p_i)$ is called almost K\"ahler--Einstein if the following conditions are satisfied.
\begin{itemize}\setlength{\itemsep}{1mm}
    \item $\Ric(\omega_i)\geq-\omega_i$.
    \item $p_i\in M_i$ and $\vol(B_{\omega_i}(p_i,1))\geq\kappa>0$ for some constant $\kappa$.
    \item There is a constant $\lambda_i\in[-1,1]$ such that the flow
\[
\begin{cases}
\,\frac{\partial \omega_i(t)}{\partial t}=-\Ric(\omega_i(t))+\lambda_i\omega_i(t), \\[1mm]
\,\omega_i(0)=\omega_i,
\end{cases}
\]
    exists on $M_i\times [0,1]$ and $\int_0^1\int_{M_i}|R(\omega_i(t))-n\lambda_i|\omega_i^ndt\to 0$.
    \item
    $\int_{M_i}|\Ric(\omega_i)-\lambda_i\omega_i|\omega^n_i\to 0$.
\end{itemize}
\end{definition}

The next result gives a typical example of almost K\"ahler--Einstein sequence, which slightly generalizes \cite[Theorem 6.2]{TW15}.
\begin{proposition}
    Let $(M_i,\omega_i)$ be a sequence of compact K\"ahler manifolds with $\omega_i\in 2\pi c_1(M_i)$,  $\Ric(\omega_i)\geq(1-\varepsilon_i)\omega_i$ for $\varepsilon_i\to 0$. Then for any $p_i\in M_i$, the sequence $(M_i,\omega_i,p_i)$ is almost K\"ahler--Einstein.
\end{proposition}

\begin{proof}
The condition $\Ric(\omega_i)\geq-\omega_i$ trivially holds. The volume non-collapsing condition is also easy to check. Indeed, by Bonnet--Mayers theorem, there exists $D=D(n)>0$ such that $\mathrm{diam}(M_i,\omega_i)< D$. Thus the Bishop--Gromov volume comparison yields that
    $$
    \frac{\vol(B_{\omega_i}(p_i,1))}{\vol(M_i,\omega_i)}=\frac{\vol(B_{\omega_i}(p_i,1))}{\vol(B_{\omega_i}(p_i,D))}\geq C^{-1}(n)>0,
    $$
    which implies that
    $$
    \vol(B_{\omega_i}(p_i,1))\geq C^{-1}(n)\int_{M_i}\omega_i^n = C^{-1}(n)(2\pi)^nc_1(M_i)^n \geq C^{-1}(n)2\pi^n.
    $$
    Here we used the algebraic fact that $c_1(M_i)^n$ is a positive integer.

We now verify the condition on the K\"ahler-Ricci flow. We choose $\lambda_i:=1$. The normalized K\"ahler-Ricci flow
\[
\begin{cases}
\,\frac{\partial \omega_i(t)}{\partial t}=-\Ric(\omega_i(t))+\omega_i(t), \\[1mm]
\,\omega_i(0)=\omega_i,
\end{cases}
\]
exists on $M_i\times [0,\infty)$. Moreover, the scalar curvature $R(\omega_i(t))$ evolves as follows:
    $$
    \frac{\partial R(\omega_i(t))}{\partial t}=\Delta_{\omega_{i}(t)} R(\omega_i(t))+|\Ric(\omega_i(t))|^2-R(\omega_i(t)).
    $$
    So one has
    $$
    \frac{\partial (R(\omega_i(t))-n)}{\partial t}\geq\Delta_{\omega_{i}(t)} (R(\omega_i(t))-n)+(R(\omega_i(t))-n).
    $$
    By the maximum principle, we obtain that
    $$
    (R(\omega_i(t))-n)\geq(R(\omega_i)-n)e^t\geq-n\varepsilon_ie^t.
    $$
    On the other hand, we have
\[
\begin{cases}
\,\frac{d[\omega_{i}(t)]}{d t}=-2\pi c_{1}(M_{i})+[\omega_i(t)], \\[1mm]
\,[\omega_i(0)] = [\omega_i] = 2\pi c_{1}(M_{i}),
\end{cases}
\]
which implies $[\omega_i(t)]=2\pi c_{1}(M_{i})$ and so
    $$
    \int_{M_i}(R(\omega_i(t))-n)\omega^n_i(t)=0\text{ for any }t\in[0,1].
    $$
It then follows that
\begin{equation*}
\begin{split}
& \int_0^1\int_{M_i}|R(\omega_i(t))-n|\omega^n_i(t)dt \\
\leq {} & \int_0^1\int_{M_i}\left(\big(R(\omega_i(t))-n+n\varepsilon_ie^t\big)+n\varepsilon_ie^t\right)\omega^n_i(t)dt \\
= {} & 2n\varepsilon_i\vol(M_i,\omega_i)\int_0^1 e^{t}dt\to 0.
\end{split}
\end{equation*}
    Here note that $\vol(M_i,\omega_i)$ has uniform upper bound, by the Bishop--Gromov volume comparison.

    Finally, we estimate
    \begin{equation*}
        \begin{aligned}
            \int_{M_i}|\Ric(\omega_i)-\omega_i|\omega^n_i&\leq\int_{M_i}\big(|\Ric(\omega_i)-(1-\varepsilon_i)\omega_i|+\varepsilon_i|\omega_i|\big)\omega^n_i\\
            &\leq\int_{M_i}\big(R(\omega_i)-(1-\varepsilon_i)n+\varepsilon_in\big)\omega^n_i\\[2mm]
            &=2n\vol(M_i,\omega_i)\varepsilon_i\to 0.
        \end{aligned}
    \end{equation*}
    So we finish the proof.
\end{proof}

Combining \cite[Theorem 2]{TW15} (see also \cite[Theorem B.1]{T15}) with Tian's partial $C^0$-estimate (cf. \cite[Theorem 5.9]{T15}), we obtain the following consequence.

\begin{corollary}\label{alke}
    Let $(M_i,\omega_i)$ be a sequence of compact K\"ahler manifolds with $\omega_i\in 2\pi c_1(M_i)$,  $\Ric(\omega_i)\geq(1-\varepsilon_i)\omega_i$ for $\varepsilon_i\to 0$. Assume that $(M_i,\omega_i)\xrightarrow{\mathrm{GH}}(Z,d)$ in the Gromov--Hausdorff topology. Then $(Z,d)$ is a compact metric space with a regular-singular decomposition: $Z=\mathcal{R}\cup\mathcal{S}$, and the following properties hold.
    \begin{itemize}\setlength{\itemsep}{1mm}
        \item There is a complex structure $J$ on $\mathcal{R}$ such that $(\mathcal{R},d,J)$ is a smooth convex open K\"ahler manifold.
        \item $\Ric(\omega_\infty)=\omega_\infty$, where $\omega_\infty$ be the induced K\"ahler form on $\mathcal{R}$.
        \item $Z$ is a normal projective variety and $(Z,d)$ is the metric completion of $(\mathcal{R},\omega_\infty)$.
        \item $\dim_{\mathcal H}\mathcal{S}\leq 2n-4$.
    \end{itemize}
\end{corollary}
Under the same condition as the above corollary, assume that $u_i$ is a sequence of functions on $(M_i,\omega_i)$ such that $\Delta_i u_i=-\lambda_i u_i, \int_{M_i} u_i^2\omega_i^n=1$, where $\lambda_i\rightarrow 1$, then $u_i$ converges to a Lipschitz function $u$ on $Z$. By \cite[Theorem 7.9]{CC00b}, $u$ is an eigenfunction on $(Z,d)$.

\begin{lemma}\label{lap}
One has $\Delta_{\omega_\infty} u=-u$ on $\mathcal{R}$ and $J \nabla u$ is a Killing vector field on $\mathcal R$ which can be extended to be the imaginary part of a holomorphic vector field on $Z$.
\end{lemma}

\begin{proof}
By the definition of Laplacian on the limit spaces in \cite{CC00b}, we know that for any Lipschitz function $\phi$,
$$\int_{\mathcal R} \langle \nabla u, \nabla \phi \rangle= \int_{\mathcal R}u\phi.$$ As in the proof of Lemma \ref{Killing}, choosing a cut-off function $\gamma_\ve$, we have
\[
\begin{split}
\int_{\mathcal R}\gamma_\ve\langle \nabla u, \nabla \phi \rangle
= {} & -\int_{\mathcal R}\gamma_\ve \phi \Delta_{\omega_\infty}  u -\int_{\mathcal R}\phi\langle \nabla u, \nabla \gamma_\ve \rangle.
\end{split}
\]
Letting $\ve\rightarrow 0$, we get
$$\int_{\mathcal R} \langle \nabla u, \nabla \phi \rangle= -\int_{\mathcal R}\phi \Delta_{\omega_\infty}  u.$$
It follows that
$$-\int_{\mathcal R}\phi \Delta_{\omega_\infty}  u=\int_{\mathcal R}u\phi. $$
which implies $\Delta_{\omega_\infty} u=-u$ on $\mathcal{R}$. Combining this with the same argument of Lemma \ref{Killing}, we complete the proof.
\end{proof}

\section{Gap Theorem}
In this section, we give the proof of Theorem \ref{thm1} and \ref{thm2}.

\begin{proof}[Proof of Theorem \ref{thm1}]
We argue by contradiction. Assume that the statement is false, then there exists a sequence of K\"ahler--Einstein manifolds $(M_i,\omega_i)$ with $\Ric(\omega_i)=\omega_i$, $\lambda_{n^2+3}^{\omega_i}\leq 1+\varepsilon_i$ for $\varepsilon_i\to 0$, such that $(M_i,\omega_i)$ is not biholomorphically isometric to $(\mathbb{CP}^n,\omega_{\mathbb{CP}^n})$ for any $i\geq1$. Up to taking a subsequence we can assume that $(M_i,\omega_i)\xrightarrow{\mathrm{GH}}(Z,d).$ As in the previous section, we have $$
Z=\mathcal{R}\cup\mathcal{S},
$$
where $\mathcal{R}$ is an open convex subset in $Z$.  Moreover, $Z$ is homeomorphic to an $n$-dimensional $\mathbb Q$-Fano variety equipped with a (singular) K\"ahler--Einstein metric $\omega_\infty$ whose restriction on $\mathcal{R}$ is a smooth K\"ahler--Einstein metric and the metric completion of $(\mathcal{R},\omega_{\infty}|_{\mathcal{R}})$ coincides with $(Z,d)$.

Let $u^{\omega_i}_1, u^{\omega_i}_2,...,u^{\omega_i}_{n^2+3}$ be the eigenfunctions of $\omega_i$ corresponding to eigenvalues $\lambda^{\omega_i}_1, \lambda^{\omega_i}_2,...,\lambda^{\omega_i}_{n^2+3}.$ Then for $j=1,2,...,n^2+3$, $u^{\omega_i}_j$ converges to functions $u_j$ on $(Z,d)$ and
\[
\Delta_{\omega_\infty} u_j=-u_j \text { on } \mathcal{R}.
\]
By Lemma \ref{Killing}, we know that $\dim_{\mathbb{R}}\mathrm{Iso}(Z,d)\geq n^2+3$.  By \cite[Theorem 4.1]{CC00a}, $\mathrm{Iso}(Z,d)$ is a  Lie group. Since $Z$ is compact, $\mathrm{Iso}(Z,d)$ is also compact. Note that the regular part $\mathcal R$ is invariant under the action of $\mathrm{Iso}(Z,d)$. By Lemma \ref{transitive lemma}, we know that the action of $\mathrm{Iso}(Z,d)$ on $\mathcal R$ is transitive. Then  $\mathcal R=\mathrm{Iso}(Z,d)\cdot p$ for any point $p\in \mathcal R$. It follows that  $\mathcal R$ is compact. Since it is dense in $Z$, it must hold that $\mathcal R=Z.$ Then $\omega_\infty$ is a smooth K\"ahler--Einstein metric on $Z$ and $\lambda_{n^2+3}=1$. By Theorem \ref{CPn rigidity}, $(Z,\omega_\infty)$ is biholomorphically isometric to $(\mathbb{CP}^n,\omega_{{\mathbb{CP}}^n})$.
Note that by \cite[p.92]{DS14} $M_i$ and $Z$ are in the same deformation family. Then we conclude that each $M_i$ is actually biholomorphic to $\mathbb{CP}^n$, thanks to the rigidity of $\mathbb{CP}^n$ \cite[Proposition 4.5]{P22}.
By the uniqueness of K\"ahler--Einstein metrics \cite{BM87} we find that $(M_i,\omega_i)$ is biholomorphically isometric to $(\mathbb{CP}^n,\omega_{{\mathbb{CP}}^n})$, which is a contradiction.
\end{proof}

\begin{proof}[Proof of Theorem \ref{thm2}]
We can argue similarly using the structure theorem of the limit space of almost K\"ahler--Einstein sequence (see Corollary \ref{alke}). Assume that the statement is false for some $\varepsilon>0$, then there exists a sequence of Fano manifolds $(M_i,\omega_i)$ with $\Ric(\omega_i)\geq (1-\delta_i)\omega_i$, $\lambda_{n^2+3}^{\omega_i}\leq 1+\delta_i$ and $\delta_i\to 0$, such that
\begin{align}\label{gh}
d_{\mathrm{GH}}\big((M_i,\omega_i),(\mathbb{CP}^n, \omega_{{\mathbb{CP}}^n})\big) \geq \varepsilon.
\end{align}
By Corollary \ref{alke}, we can assume that $Z=\mathcal{R}\cup\mathcal{S}$ and $Z$ has the same property as above. By Lemma \ref{lap},  we also know that $\dim_{\mathbb{R}}\mathrm{Iso}(Z,d)\geq n^2+3$. Then by the proof of Theorem \ref{thm1}, $(Z,d)$ is biholomorphically isometric to $(\mathbb{CP}^n,\omega_{{\mathbb{CP}}^n})$, which contradicts \eqref{gh}.
\end{proof}

\section{Examples}\label{examples}
We give several simple examples showing that the main results of our paper are optimal.
We first recall the following standard result.

\begin{lemma}
\label{lem:Cpn-multiplicity}
    The first eigenvalue of $(\mathbb{CP}^n,\omega_{\mathbb{CP}^n})$ has multiplicity $n^2+2n$. In other words, $\lambda_1=...=\lambda_{n^2+2n}=1$ and $\lambda_{n^2+2n+1}>1$.
\end{lemma}

\begin{proof}
This lemma follows from Lemma \ref{isomorphism} and
\[
\dim_\mathbb{R} \mathrm{Iso}(\mathbb{CP}^n,\omega_{\mathbb{CP}^n})=\dim_\mathbb{R}PU(n+1)=n^2+2n.
\]
\end{proof}

In fact, we can explicitly describe all the eigenfunctions of $\lambda_1$ on $(\mathbb{CP}^n,\omega_{\mathbb{CP}^n})$. Consider the smooth embedding
\begin{align*}
\varphi:\mathbb{CP}^n&\rightarrow H_{n+1}\\
[Z_0,...,Z_n]&\mapsto \bigg(\frac{Z_i\bar Z_j}{|Z|^2}\bigg),
\end{align*}
where $H_{n+1}$ denotes the space of $(n+1)\times(n+1)$ Hermitian matrices. For any traceless $A\in H_{n+1}$, define a function $f_A$ on $\mathbb{CP}^n$ by
$$
f_A([Z]):=\mathrm{tr}(A\varphi([Z])),\ [Z]\in\mathbb{CP}^n.
$$
Then $f_A$ is an eigenfunction of $\lambda_1$ and any eigenfunction is of this form.

Now, for positive integers $k$, $l$ such that $k+l=n$ and constant $a\in[0,1)$. Consider
\[
(M,\omega) =
(\mathbb{CP}^{k},\omega_{\mathbb{CP}^{k}})\times\big(\mathbb{CP}^{l},(1-a)\omega_{\mathbb{CP}^{l}}\big).
\]
Denote the the projection map by $\pi_{i}$ ($i=1,2$). Then
\[
\begin{split}
\Ric(\omega) = {} & \pi_{1}^{*}\Ric(\omega_{\mathbb{CP}^{k}})+\pi_{2}^{*}\Ric(\omega_{\mathbb{CP}^{l}}) \\[1mm]
= {} & \pi_{1}^{*}\omega_{\mathbb{CP}^{k}}+\pi_{2}^{*}\omega_{\mathbb{CP}^{l}} \\[1mm]
= {} & \omega+a\pi_{2}^{*}\omega_{\mathbb{CP}^{l}}\geq\omega. \\[1.5mm]
\end{split}
\]
Combining Lemma \ref{lem:Cpn-multiplicity} with Lemma \ref{example lemma a zero} and \ref{example lemma a non-zero} (see below), we see that $\lambda_{1}(M,\omega)=1$, and the multiplicity is $k^{2}+2k+l^{2}+2l$ when $a=0$ and $k^{2}+2k$ when $a\in(0,1)$.

\begin{lemma}\label{example lemma a zero}
When $a=0$, the followings are equivalent:
\begin{enumerate}\setlength{\itemsep}{1mm}
\item[(a)] $u$ is the eigenfunction of $1$ on $M$.
\item[(b)] $u=\pi_{1}^{*}v_{1}+\pi_{2}^{*}v_{2}$, where $v_{1}$ and $v_{2}$ are the eigenfunctions of $1$ on $\mathbb{CP}^{k}$ and $\mathbb{CP}^{l}$ respectively.
\end{enumerate}
\end{lemma}

\begin{proof}
For $(a)\Rightarrow(b)$, using Lemma \ref{isomorphism} and $\theta=0$,
\[
J\nabla u \in \mathfrak{h}(M,\omega) \cong \mathfrak{h}(\mathbb{CP}^{k},\omega_{\mathbb{CP}^{k}})\oplus\mathfrak{h}(\mathbb{CP}^{l},\omega_{\mathbb{CP}^{l}}).
\]
Then there exists $v_{1}\in\Lambda_{1}(\mathbb{CP}^{k},\omega_{\mathbb{CP}^{k}})$ and $v_{2}\in\Lambda_{1}(\mathbb{CP}^{l},\omega_{\mathbb{CP}^{l}})$ such that
\[
J\nabla u = J\nabla(\pi_{1}^{*}v_{1})+J\nabla(\pi_{2}^{*}v_{2})
\]
and so $u=\pi_{1}^{*}v_{1}+\pi_{2}^{*}v_{2}$. The converse direction $(b)\Rightarrow(a)$ can be verified by direct calculation.
\end{proof}

\begin{lemma}\label{example lemma a non-zero}
When $a\in(0,1)$, the followings are equivalent:
\begin{enumerate}\setlength{\itemsep}{1mm}
\item[(a)] $u$ is the eigenfunction of $1$ on $M$.
\item[(b)] $u=\pi_{1}^{*}v$, where $v$ is the eigenfunction of $1$ on $\mathbb{CP}^{k}$.
\end{enumerate}
\end{lemma}

\begin{proof}
For $(a)\Rightarrow(b)$, using Lemma \ref{eigenfunction lemma} (b), we obtain $\iota_{J\nabla u}\theta=0$, where $\theta=\Ric(\omega)-\omega=a\pi_{2}^{*}\omega_{\mathbb{CP}^{l}}$. It follows that $(\pi_{2})_{*}(J\nabla u)=0$ and so $(\pi_{2})_{*}(\nabla u)=0$. Then there exists a function $v$ on $\mathbb{CP}^{k}$ such that $u=\pi_{1}^{*}v$. Since $u$ is the eigenfunction of $1$ on $M$, we obtain
\begin{equation}\label{example lemma eqn}
\pi_{1}^{*}(\Delta_{\mathbb{CP}^{k}}v+v) = \Delta_{M}u+u = 0.
\end{equation}
which implies that $v$ is the eigenfunction of $1$ on $\mathbb{CP}^{k}$. The converse direction $(b)\Rightarrow(a)$ can be verified by direct calculation.
\end{proof}

\section{Further discussions}
\label{sec:2&3-large-mult}

Theorem \ref{CPn rigidity} shows that $(\mathbb{CP}^{n},\omega_{\mathbb{CP}^{n}})$ has the largest multiplicity $n^2+2n$ of the first eigenvalue among all K\"ahler manifolds with $\Ric(\omega)\geq\omega$ and $\lambda_1=1$, and there is nothing else with multiplicity in the range $[n^2+3,n^2+2n)$. It is then natural to ask which manifold achieves the second largest multiplicity $n^{2}+2$. The examples in Section \ref{examples} suggest that the answer should be the product space $\mathbb{CP}^{n-1}\times\mathbb{CP}^1$. To push even further, one could ask which manifold attains the third largest multiplicity $n^2+1$.

Our next result gives complete answers to these questions.
The proof relies on some general classification results for complex manifolds admitting biholomorphic isometry groups of dimension $n^2+2$ and $n^2+1$ \cite{I07,IK09}.

\begin{theorem}\label{CPn-1 CP1 rigidity}
Let $(M,\omega)$ be a compact K\"ahler manifold of complex dimension $n\geq2$ with $\Ric(\omega)\geq\omega$.
\begin{enumerate}\setlength{\itemsep}{1mm}
    \item[(a)] If $\lambda_{n^2+2}=1$ and $\lambda_{n^2+3}>1$, then $(M,\omega)$ is biholomorphically isometric to $(\mathbb{CP}^{n-1},\omega_{\mathbb{CP}^{n-1}})\times(\mathbb{CP}^1,\omega_{\mathbb{CP}^1})$.
    \item[(b)] If $\lambda_{n^2+1}=1$ and $\lambda_{n^2+2}>1$, then $n=3$ and $(M,\omega)$ is biholomorphically isometric to the quadric hypersurface in $\mathbb{CP}^4$ equipped with its canonical K\"ahler--Einstein metric.
\end{enumerate}
\end{theorem}

\begin{proof}
In either case, one has $\lambda_{n^2+1}=1$, it then follows from Proposition \ref{KE rigidity} that $\Ric(\omega)=\omega$, namely $(M,\omega)$ is a K\"ahler--Einstein Fano manifold.
	
	We can further determine the complex structure of $M$, thanks to the classification result in \cite{I07,IK09}. Indeed, note that the group $G:=\mathrm{Iso}(M,\omega)$ acts properly on $M$, in the sense that the map
	$$
	\Psi: G\times M\to M\times M,\ (g,p)\mapsto(gp,p),
	$$
	is proper, i.e., for every compact subset $C\subset M\times M$ its inverse image $\Psi^{-1}(C)\subset G\times M$ is compact as well. In case (a) we have $\dim_{\mathbb{R}}G=n^2+2$, while in case (b) we have $\dim_{\mathbb{R}}G=n^2+1$.
	Going through the classification list in \cite{I07,IK09} and using that our underlying manifold is Fano, we find that $M$ is bihilomorphic to $\mathbb{CP}^{n-1}\times\mathbb{CP}^1$ in case (a) (by \cite[Theorem 3.1]{I07}), and to the quadric in $\mathbb{CP}^4$ in case (b) (by \cite[Proposition 1.1 and Theorem 2.1]{IK09}). Note that both $\mathbb{CP}^{n-1}\times\mathbb{CP}^1$ and the quadric can be equipped with a K\"ahler--Einstein metric, which is unique up to automorphisms \cite{BM87}. So $M$ is further isometric to one of them, which then completes the proof.
	\end{proof}

Under the K\"ahler--Einstein assumption,	
we can further establish the corresponding almost rigidity result, sharpening our Theorem \ref{thm1}.

\begin{theorem}
	There exists a dimensional constant $\varepsilon=\varepsilon(n)>0$ such that the following holds. Let $(M,\omega)$ be a compact K\"ahler manifold of complex dimension $n\geq2$ with $\Ric(\omega)=\omega$.
\begin{enumerate}\setlength{\itemsep}{1mm}
\item[(a)] If $\lambda_{n^2+3}< 1+\varepsilon$, then $(M,\omega)$ is biholomorphically isometric to $(\mathbb{CP}^n,\omega_{\mathbb{CP}^n})$.
\item[(b)] If $\lambda_{n^2+2}< 1+\varepsilon$ and $\lambda_{n^2+3}>1$, then $(M,\omega)$ is biholomorphically isometric to $(\mathbb{CP}^{n-1},\omega_{\mathbb{CP}^{n-1}})\times(\mathbb{CP}^1,\omega_{\mathbb{CP}^1})$.
\item[(c)] If $\lambda_{n^2+1}< 1+\varepsilon$ and $\lambda_{n^2+2}>1$, then $n=3$ and $(M,\omega)$ is biholomorphically isometric to the quadric hypersurface in $\mathbb{CP}^4$ equipped with its canonical K\"ahler--Einstein metric.
\end{enumerate}
\end{theorem}

\begin{proof}
It suffices to prove (b) and (c), as (a) is already settled in Theorem \ref{thm1}.
The proof for these two cases will be almost the same as that for (a), so we will be sketchy. We argue by contradiction. After passing to the limit space $Z=\mathcal{R}\cup\mathcal S$, we claim that one still has $Z=\mathcal R$. Indeed, note that $\dim_{\mathbb R}\mathrm{Iso}(Z,d)\geq n^2+1$. So the action of $\mathrm{Iso}(Z,d)$ on $\mathcal R$ is transitive by Lemma \ref{transitive lemma}, which implies that $Z=\mathcal R$, as desired. Therefore, the limit space $(Z,d)$ is a K\"ahler--Einstein Fano manifold. Using \cite{I07,IK09} again we conclude that $Z$ is biholomorphic to one of the following two possibilities:
	\begin{enumerate}\setlength{\itemsep}{1mm}
		\item $\mathbb{CP}^{n-1}\times\mathbb{CP}^1$.
		\item the quadric in $\mathbb{CP}^4$.
	\end{enumerate}
Note that $\mathbb{CP}^{n-1}\times\mathbb{CP}^1$ is rigid under deformations \cite[Proposition 4.5]{P22}, thus proving (b). The quadric is also rigid, by \cite[Proposition 4.2]{BB96}, which then proves (c).
\end{proof}

\end{document}